\documentclass[11pt]{article}
\usepackage[dvips]{graphicx}
\usepackage{color}
\usepackage{amsmath}
\usepackage{amsfonts}
\usepackage{amssymb}
\usepackage{amsthm}
\usepackage{newlfont}
\usepackage{epsfig}
\usepackage{multirow}
\usepackage{setspace}
\usepackage{hyperref}
\usepackage{cite}

\begin{document}

\newtheorem{assumption}{Assumption}[section]
\newtheorem{definition}{Definition}[section]
\newtheorem{lemma}{Lemma}[section]
\newtheorem{proposition}{Proposition}[section]
\newtheorem{theorem}{Theorem}[section]
\newtheorem{corollary}{Corollary}[section]
\newtheorem{remark}{Remark}[section]

\title{Dynamical Equivalence and Linear Conjugacy of Chemical Reaction Networks: New Results and Methods}
\author{Matthew D. Johnston$^a$, David Siegel$^a$ and G\'{a}bor Szederk\'{e}nyi$^{b,c}$ \bigskip \\
${}^a$ Department of Applied Mathematics,\\
University of Waterloo,\\
Waterloo, Ontario, Canada N2L 3G1\\
${}^b$ (Bio)Process Engineering Group, IIM-CSIC,\\ Spanish National Research Council,\\ C/Eduardo Cabello, 6, 36208 Vigo, Spain \\
${}^c$ Computer and Automation Research Institute,\\
Hungarian Academy of Sciences\\
H-1518, P.O. Box 63, Budapest, Hungary}
\date{}
\maketitle

\tableofcontents

\bigskip

\begin{abstract}

In the first part of this paper, we propose new optimization-based methods for the computation of preferred (dense, sparse, reversible, detailed and complex balanced) linearly conjugate reaction network structures with mass action dynamics. The developed methods are extensions of previously published results on dynamically equivalent reaction networks and are based on mixed-integer linear programming. As related theoretical contributions we show that (i) dense linearly conjugate networks define a unique super-structure for any positive diagonal state transformation if the set of chemical complexes is given, and (ii) the existence of linearly conjugate detailed balanced and complex balanced networks do not depend on the selection of equilibrium points. In the second part of the paper it is shown that determining dynamically equivalent realizations to a network that is structurally fixed but parametrically not can also be written and solved as a mixed-integer linear programming problem. Several examples illustrate the presented computation methods.

\end{abstract}


\bigskip

\section{Introduction}

The mathematical study of chemical reaction networks is a rapidly growing field that has been applied recently to research problems in industrial chemistry, systems biology, gene regulation, and general nonlinear systems theory, among others \cite{C-P-F,H2,S}. There has also been significant theoretical work in the literature on such topics as persistence \cite{A,A-S,A3,J-S1,J-S3}, multistability \cite{C-F1, C-F2,S-F}, monotonicity \cite{B-A,B2}, the global attractor conjecture for complex balanced systems \cite{A,A-S,C-D-S-S,C-N-P}, lumping \cite{Fa,L-R1,T-L-R-T,W-K}, and conjugacy of reaction networks \cite{C-P,J-S2}.

One line of research which has been attracting increased attention has been that of determining when two reaction networks exhibit the same qualitative dynamics despite disparate network structure. In \cite{C-P} and \cite{Sz1}, the authors complete the question of what network structures can give rise to the same set of governing differential equations and therefore exhibit identical dynamics. This work was extended in \cite{J-S2} to networks which do not necessarily have the same set of differential equations but rather have trajectories related by a non-trivial linear transformation. Similar ground has also been touched in the papers \cite{L-R1} and \cite{Fa} which deal with properties of linear lumpings, linear mappings which potentially reduce the dimension of the species set.

In this paper we look at the problem of algorithmically determining when a network is linearly conjugate to another network satisfying specified conditions. This problem was first addressed in \cite{Sz2} where the author presents a mixed-integer linear programming (MILP) algorithm capable of determining sparse and dense realizations, i.e. networks with the fewest and greatest number of reactions capable of generating the given dynamics. The algorithm was extended to complex and detailed balanced networks in \cite{Sz-H}, reversible networks in \cite{Sz-H-P}, weakly reversible networks in \cite{Sz-H-T}, and linear conjugate networks in \cite{J-S4}, where a computationally more efficient procedure for determining weak reversibility was also presented.

We will continue in this paper the development and application of this MILP framework to linearly conjugate networks. In particular, we will give conditions which can be used to find sparse and dense linearly conjugate networks which are detailed and complex balanced, fully reversible, and which contain the greatest and fewest number of complexes. We will also expand the original MILP algorithm for finding sparse and dense realizations to find alternative realizations to a given reaction network when the network structure is fixed but the parameter values are not. Since this algorithm works without having to have the rate constants specified beforehand, this will allow us to answer questions about the reaction mechanism itself.

\section{Background}

In this section we present the terminology and notation relevant to chemical reaction networks and the main results from the literature upon which we will be building.

\subsection{Chemical Reaction Networks}

We will consider the chemical \emph{species} or \emph{reactants} of a network to be given by the set $\mathcal{S} = \left\{ X_1, X_2, \ldots, X_n \right\}$. The combined elements on the left-hand and right-hand side of a reaction are given by linear combinations of these species. These combined terms are called \emph{complexes} and will be denoted by the set $\mathcal{C} = \left\{ C_1, C_2, \ldots, C_m \right\}$ where
\[C_i = \sum_{j=1}^n \alpha_{ij} X_j, \; \; \; i=1, \ldots, m\]
and the $\alpha_{ij}$ are nonnegative integers called the \emph{stoichiometric coefficients}. We define the reaction set to be $\mathcal{R} = \left\{ (C_i,C_j) \; | \; C_i \mbox{ reacts to form } C_j \right\}$ where the property $(C_i,C_j) \in \mathcal{R}$ will more commonly be denoted $C_i \to C_j$. To each $(C_i,C_j) \in \mathcal{R}$ we will associate a positive \emph{rate constant} $k(i,j) > 0$ and to each $(C_i,C_j) \not\in \mathcal{R}$ we will set $k(i,j) = 0$. The triplet $\mathcal{N} = (\mathcal{S}, \mathcal{C}, \mathcal{R})$ will be called the \emph{chemical reaction network}.

The above formulation naturally gives rise to a directed graph $G(V,E)$ where the set of vertices is given by $V = \mathcal{C}$, the set of directed edges is given by $E = \mathcal{R}$, and the rate constants $k(i,j)$ correspond to the weights of the edges from $C_i$ to $C_j$. In the literature this has been termed the \emph{reaction graph} of the network \cite{H-J1}. Since complexes may be involved in more than one reaction, as a product or a reactant, there is further graph theory we may consider. A \emph{linkage class} is a maximally connected set of complexes, that is to say, two complexes are in the same linkage class if and only if there is a sequence of reactions in the reaction graph (of either direction) which connects them. A reaction network is called \emph{reversible} if $C_i \to C_j$ for any $C_i, C_j \in \mathcal{C}$ implies $C_j \to C_i$. A reaction network is called \emph{weakly reversible} if $C_i \to C_j$ for any $C_i, C_j \in \mathcal{C}$ implies there is some sequence of complexes such that $C_i = C_{\mu(1)} \to C_{\mu(2)} \to \cdots \to C_{\mu(l-1)} \to C_{\mu(l)} = C_j$.

A directed graph is called \emph{strongly connected} if there exists a directed path from each vertex to every other vertex. A \emph{strongly connected component} of a directed graph is a maximal set of vertices for which paths exists from each vertex in the set to every other vertex in the set. For a weakly reversible network, the linkage classes clearly correspond to the strongly connected components of the reaction graph.

Assuming mass-action kinetics, the dynamics of the specie concentrations over time is governed by the set of differential equations
\begin{equation}
\label{de}
\frac{d\mathbf{x}}{dt} = Y \cdot A_k \cdot \Psi(\mathbf{x})
\end{equation}
where $\mathbf{x} = [ x_1 \; x_2 \; \cdots \; x_n ]^T$ is the vector of reactant concentrations. The \emph{stoichiometric matrix} $Y$ contains entries $[Y]_{ij} = \alpha_{ji}$ and the \emph{Kirchhoff} or \emph{kinetics} matrix $A_k$ is given by
\begin{equation}
\label{kinetics}
[A_k]_{ij} = \left\{ \begin{array}{cll} -\sum_{l=1,l \not= i}^m k(i,l), & \mbox{  if  } & i = j \\ k(j,i) & \mbox{  if  } & i \not= j. \end{array} \right.
\end{equation}
When we speak of the \emph{structure} of a kinetics matrix, we will be referring to the distribution of positive and zero entries, which determines the network structure of the corresponding reaction graph. Finally, the \emph{mass-action vector} $\Psi(\mathbf{x})$ is given by
\begin{equation}
\label{psi}
\Psi_j(\mathbf{x}) = \prod_{i=1}^n x_i^{[Y]_{ij}}, \; \; \; j=1, \ldots, m.
\end{equation}

\subsection{Linearly Conjugate Networks}

Under the assumption of mass-action kinetics, it is possible for the trajectories of two reaction networks $\mathcal{N}$ and $\mathcal{N}'$ to be related by a linear transformation and therefore share many of the same qualitative features (e.g. number and stability of equilibria, persistence/extinction of species, dimensions of invariant spaces, etc.). This phenomenon was termed \emph{linear conjugacy} in \cite{J-S2}.

For completeness, we include the formal definition of linear conjugacy as presented in \cite{J-S2}. We will let $\Phi(\mathbf{x}_0,t)$ denote the flow of (\ref{de}) associated with $\mathcal{N}$ and $\Psi(\mathbf{x}_0,t)$ denote the flow of (\ref{de}) associated with $\mathcal{N}'$.

\begin{definition}
\label{conjugate1}
We will say two chemical reaction networks $\mathcal{N}$ and $\mathcal{N}'$ are \textbf{linearly conjugate} if there exists a linear function $\mathbf{h}: \mathbb{R}^n_{>0} \mapsto \mathbb{R}_{>0}^n$ such that $\mathbf{h}(\Phi(\mathbf{x}_0,t))=\Psi(\mathbf{h}(\mathbf{x}_0),t)$ for all $\mathbf{x}_0 \in \mathbb{R}_{>0}^n$.
\end{definition}
\noindent It is known that linear transformations $\mathbf{h}: \mathbb{R}^n_{>0} \mapsto \mathbb{R}_{>0}^n$ can consist of at most positive scaling and reindexing of coordinates (Lemma 3.1, \cite{J-S2}). Linear conjugacy has been subsequently studied from a computational point of view in \cite{J-S4}.

Linear conjugacy is a generalization of the concept of \emph{dynamical equivalence} whereby two reaction networks with different topological network structure can generate the same exact set of differential equations (\ref{de}). Two dynamically equivalent networks $\mathcal{N}$ and $\mathcal{N}'$ are said to be alternate \emph{realizations} of the kinetics (\ref{de}), although it is sometimes preferable to say that $\mathcal{N}$ is an alternative realization of $\mathcal{N}'$ or vice-versa. Since the case of two networks being realizations of the same kinetics is encompassed as a special case of linear conjugacy taking the transformation to be the identity, we will focus on linearly conjugate networks.

In general practice, we are given a network $\mathcal{N}$ and asked to determine its dynamical behaviour. This is often a challenging problem; however, we may notice that the network $\mathcal{N}$ \emph{behaves like} one from a well-studied class of networks and therefore suspect a relationship which preserves key qualitative aspects of the dynamics. The theory of linear conjugacy can provide a powerful tool in analyzing such networks. If the network can be shown to be linearly conjugate to a network $\mathcal{N}'$ from the class of networks with understood dynamics, the dynamics of $\mathcal{N}'$ are transferred to $\mathcal{N}$.

This raises the question of how to find a linearly conjugate network $\mathcal{N}'$ when only the original network $\mathcal{N}$ is given. This was studied in \cite{J-S4} where the authors built upon the linear programming algorithm introduced in \cite{Sz2}. We can impose that a network $\mathcal{N}'$ be linearly conjugate to our given network $\mathcal{N}$ with the set of linear constraints
\begin{equation}
\label{conjugate}
\mbox{\textbf{(LC)}} \; \; \left\{ \; \; \begin{array}{ll} & \displaystyle{Y \cdot A_b = T^{-1} \cdot M} \\ & \displaystyle{\sum_{i=1}^m [A_b]_{ij} = 0, \; \; \; j=1, \ldots, m} \\ & \displaystyle{[A_b]_{ij} \geq 0, \; \; \; i,j = 1, \ldots, m, \; i \not= j} \\ & \displaystyle{[A_b]_{ii} \leq 0, \; \; \; i = 1, \ldots, m} \\ &\displaystyle{\epsilon  \leq c_j \leq 1/\epsilon, \; \; \; j=1, \ldots, n} \end{array} \right.
\end{equation}
where $0 < \epsilon \ll 1$, and the matrices $M \in \mathbb{R}^{n \times m}$ and $T \in \mathbb{R}^{n \times n}$ are given by:
\begin{eqnarray}
\label{M} M & = & Y \cdot A_k, \mbox{ and} \\
\label{T} T & = & \mbox{diag}\left\{ \mathbf{c} \right\}.
\end{eqnarray}

The kinetics matrix for the network $\mathcal{N}'$ can by constructed from $A_b \in \mathbb{R}^{m \times m}$ and $\mathbf{c} \in \mathbb{R}_{>0}^n$ by the relation
\begin{equation}
\label{newrateconstants}
A_k' = A_b \cdot \mbox{diag} \left\{ \Psi (\mathbf{c}) \right\}.
\end{equation}
Finding a network satisfying (\ref{conjugate}) and then solving (\ref{newrateconstants}) is sufficient to determine a network $\mathcal{N}'$ which is linearly conjugate to $\mathcal{N}$ via the transformation $\mathbf{h}(\mathbf{x}) = T^{-1} \mathbf{x}$.

\subsection{Sparse and Dense Linearly Conjugate Networks}
\label{section1}


In order to place the problem within a linear programming framework, we need to choose an objective function to optimize. An appropriate choice of such a function is not obvious and may vary depending on the intended application.

One particularly intuitive choice, which was introduced in \cite{Sz2} and has been widely used since, is to search for networks $\mathcal{N}'$ with the fewest and greatest number of reactions (\emph{sparse} and \emph{dense} networks, respectively). A sparse (respectively, dense) linearly conjugate network is given by a matrix $A_k'$ satisfying (\ref{conjugate}) with the most (respectively, least) off-diagonal entries which are zeroes. Since the structure of $A_k'$ and $A_b$ are the same, a correspondence between the non-zero off-diagonal entries in $A_k'$ and a positive integer value can be made by considering the binary variables $\delta_{ij} \in \left\{ 0, 1 \right\}$ which will keep track of whether a reaction is `on' or `off', i.e. we have
\[\delta_{ij} = 1 \Longleftrightarrow [A_b]_{ij} > \epsilon, \; \; \; i, j = 1, \ldots, m, \; \; i \not= j\]
where $0 < \epsilon \ll 1$ is sufficiently small and can be chosen the same as in (\ref{conjugate}), and where the symbol `$\Longleftrightarrow$' denotes the logical relation `if and only if'. These proposition logic constraints for the structure of a network can then be formulated as the following linear mixed-integer constraints (see, for example, \cite{R-G}):
\begin{equation}
\label{density}
\mbox{\textbf{(S)}} \; \; \left\{ \; \; \begin{array}{ll} & \displaystyle{0 \leq [A_b]_{ij}-\epsilon \delta_{ij}, \; \; \; i,j = 1, \ldots, m, \; \; i \not= j} \\ & \displaystyle{0 \leq -[A_b]_{ij}+u_{ij} \delta_{ij}, \; \; \; i,j = 1, \ldots, m, \; \; i \not= j} \\ & \displaystyle{\delta_{ij} \in \left\{ 0, 1 \right\}, \; \; \; i,j = 1, \ldots, m, i \not= j,} \end{array} \right.
\end{equation}
where $u_{ij} > 0$ for $i, j = 1, \ldots, m, i \not= j,$ are appropriate upper bounds for the reaction rate coefficients. The number of reactions present in the network corresponding to $A_k$ is then given by the sum of the $\delta_{ij}$'s so that the problem of determining a sparse network corresponds to satisfying the objective function
\begin{equation}
\label{sparse}
\mbox{\textbf{(Sparse)}} \; \; \left\{ \; \; \; \; \; \; \; \; \mbox{minimize} \; \; \; \; \; \sum_{i,j=1, i \not= j}^m \delta_{ij} \right.
\end{equation}
over the constraint sets (\ref{conjugate}) and (\ref{density}). Finding a dense network corresponds to maximizing the same function, which can also be stated as a minimization problem as
\begin{equation}
\label{dense}
\mbox{\textbf{(Dense)}} \; \; \left\{ \; \; \; \; \; \; \; \; \mbox{minimize} \; \; \; \; \; \sum_{i,j=1, i \not= j}^m -\delta_{ij}. \right.
\end{equation}

It is known that, for trivial linear conjugacies, the structure of the dense realization contains the structure of all other trivial linear conjugacies as a subset (Theorem 3.1, \cite{Sz-H-P}). We now prove the comparable result for non-trivial linear conjugacies.

\begin{theorem}
\label{densestructure}
Let $\mathcal{N}$ be a chemical reaction network. Suppose that the reaction network $\mathcal{N}'$ is linearly conjugate to $\mathcal{N}$ and dense. Suppose that $\mathcal{N}''$ is also linearly conjugate to $\mathcal{N}$. Then the directed unweighted graph of $\mathcal{N}''$ is a subset of the directed unweighted graph of $\mathcal{N}'$.
\end{theorem}

\begin{proof}
Assume $\mathcal{N}'$ and $\mathcal{N}''$ are linearly conjugate to $\mathcal{N}$, $\mathcal{N}'$ is dense in the space of networks which are linearly conjugate to $\mathcal{N}$, and $\mathcal{N}''$ contains a reaction not contained in $\mathcal{N}'$.

Since both $\mathcal{N}'$ and $\mathcal{N}''$ are linearly conjugate to $\mathcal{N}$, we have by (\ref{conjugate}) that
\[Y \cdot A_k = T' \cdot Y \cdot A_b'\]
and
\[Y \cdot A_k = T'' \cdot Y \cdot A_b''\]
where $T' = \mbox{diag} \left\{ \mathbf{c}' \right\}$, $\mathbf{c}' \in \mathbb{R}_{>0}^n$ and $A_b'$ correspond to $\mathcal{N}'$, $\mathbf{c}'' \in \mathbb{R}_{>0}^n$ and $T'' = \mbox{diag} \left\{ \mathbf{c}'' \right\}$ and $A_b''$ correspond to $\mathcal{N}''$.

Now consider $T' \cdot Y \cdot A_b''$. We have
\[T' \cdot Y \cdot A_b'' = T' \cdot (T'')^{-1} \cdot T'' \cdot Y \cdot A_b'' = Q \cdot Y \cdot A_k\]
where $Q = T' \cdot (T'')^{-1}$. Consequently, we have
\begin{equation}
\label{25}
T' \cdot Y \cdot A_b'' + T' \cdot Y \cdot A_b' = (Q + I) \cdot Y \cdot A_k.
\end{equation}
On the other hand, we have
\begin{equation}
\label{26}
T' \cdot Y \cdot A_b'' + T' \cdot Y \cdot A_b' = T' \cdot Y \cdot (A_b'' + A_b') = T' \cdot Y \cdot A_b'''
\end{equation}
where $A_b''' = A_b' + A_b''$. Combining (\ref{25}) and (\ref{26}) gives
\[Y \cdot A_k = (Q + I)^{-1} \cdot T' \cdot Y \cdot A_b''' = T''' \cdot Y \cdot A_b'''\]
where $T''' = (T' \cdot (T'')^{-1} + I)^{-1} \cdot T'$.

Since $T'$, $T''$ and $I$ are diagonal matrices with positive entries on the diagonal, so is $T'''$. This means that the network $\mathcal{N}'''$ corresponding to $A_b''' = A_b' + A_b''$ is linearly conjugate to $\mathcal{N}$. We can readily see, however, that $\mathcal{N}'''$ contains all of the reactions in both $\mathcal{N}'$ and $\mathcal{N}''$. If $\mathcal{N}''$ contains a reaction not contained in $\mathcal{N}'$ then $\mathcal{N}'''$ clearly has more reactions than $\mathcal{N}'$ which contradicts the assumption that $\mathcal{N}'$ is dense in the space of networks which are linearly conjugate to $\mathcal{N}$. The result follows.
\end{proof}

The following result follows immediately.

\begin{corollary}
\label{densecorollary}
Let $\mathcal{N}$ be a chemical reaction network. Then the structure of the unweighted directly graph of the dense reaction network $\mathcal{N}'$ which is linearly conjugate to $\mathcal{N}$ is unique.
\end{corollary}

\begin{proof}
This follows directly from Theorem \ref{densestructure}.
\end{proof}

\section{Computational Extensions of Linearly Conjugate Networks}
\label{extensions}

In this section we will extend the optimization framework introduced in Section \ref{section1} to include complex balanced, reversible and detailed balanced networks, and to search for networks with the greatest and fewest number of complexes.

\subsection{Weakly Reversible Networks}
\label{wrsection}

Weakly reversible networks are a particular important class of reaction networks because strong properties are known about their dynamics and equilibrium concentrations. Under a supplemental condition on the rate constants, weakly reversible networks are known to have complex balanced equilibrium concentrations and therefore exhibit all of the dynamical properties normally reserved for these networks \cite{F1,H} (see Section \ref{cbsection} for further discussion of complex balanced networks).

Consequently, they are a primary candidate for the type of network we would like to find. The problem of determining if and when a chemical reaction network $\mathcal{N}$ is linearly conjugate to a weakly reversible network $\mathcal{N}'$ was first considered in \cite{Sz-H-T} and further refined in \cite{J-S4}. For convenience, we briefly recall the constraints published in \cite{J-S4} that guarantee the weak reversibility of the linearly conjugate network $\mathcal{N}'$:
\begin{equation}
\label{weakreversibility}
\mbox{\textbf{(WR)}} \; \; \left\{ \; \; \begin{array}{ll} & \displaystyle{\sum_{i=1,i \not= j}^m[\tilde{A}_k]_{ij} = \sum_{i=1,i \not= j}^m[\tilde{A}_k]_{ji}, \; \; \; j=1, \ldots, m}\\ & \displaystyle{[\tilde{A}_k]_{ij} \geq 0, \; \; \; i,j = 1, \ldots, m, \; i \not= j} \end{array} \right.
\end{equation}
where $\tilde{A}_k$ is an auxiliary Kirchhoff matrix with the same structure as $A_k'$ with appropriately scaled columns such that its kernel contains the $m$-dimensional vector of all ones. In order to guarantee that the matrix $\tilde{A}_k$ has the same structure as $A_k'$ and $A_b$ we also require that
\begin{equation}
\label{density2}
\mbox{\textbf{(WR-S)}} \; \; \left\{ \; \; \begin{array}{ll} & \displaystyle{0 \leq [\tilde{A}_k]_{ij}-\epsilon \delta_{ij}, \; \; \; i,j = 1, \ldots, m, \; \; i \not= j}\\ & \displaystyle{0 \leq -[\tilde{A}_k]_{ij}+u_{ij} \delta_{ij}, \; \; \; i,j = 1, \ldots, m, \; \; i \not= j}. \end{array} \right.
\end{equation}

\subsection{Reversible Networks}

In \cite{Sz-H}, an algorithm was presented which was capable of determining reversible reaction networks which are trivially linearly conjugate to a given reaction network. In this section, we present a simplified methodology and apply it to non-trivial linear conjugacies.

We recall that a network is reversible if $\mathcal{C}_i \to \mathcal{C}_j$ for any $\mathcal{C}_i, \mathcal{C}_j \in \mathcal{C}$ implies $\mathcal{C}_j \to \mathcal{C}_i$. For the network $\mathcal{N}'$, this is equivalent to the condition
\[[A_k']_{ij} > \epsilon \; \Longleftrightarrow \; [A_k']_{ji} > \epsilon\]
for some sufficient small $0 < \epsilon \ll 1$. This is in turn equivalent to
\[\delta_{ij} = 1 \; \Longleftrightarrow \; \delta_{ji} = 1\]
where $\delta_{ij} \in \left\{ 0, 1 \right\}$, $i, j = 1, \ldots, m,$ $i \not= j$, as in Section \ref{section1}. It follows that we can restrict our search space to reversible networks with the constraint set
\begin{equation}
\label{reversible}
\mbox{\textbf{(Rev)}} \; \; \left\{ \; \; \begin{array}{l} \delta_{ij} - \delta_{ji} = 0, \\ i,j = 1, \ldots, m, \; \; i < j. \end{array} \right.
\end{equation}
A sparse or dense reversible network which is linearly conjugate to $\mathcal{N}$ can be found by optimizing (\ref{sparse}) or (\ref{dense}), respectively, over the constraint sets (\ref{conjugate}), (\ref{density}) and (\ref{reversible}).

\subsection{Complex Balanced Systems}
\label{cbsection}

A particularly important class of chemical reaction networks are the complex balanced networks introduced in \cite{H-J1}.
\begin{definition}
An equilibrium concentration $\mathbf{x}^* \in \mathbb{R}_{>0}^n$ of the chemical reaction network $\mathcal{N}$ is a \textbf{complex balanced equilibrium concentration} if
\begin{equation}
\label{cb}
A_k \cdot \Psi(\mathbf{x}^*) = \mathbf{0}.
\end{equation}
The network $\mathcal{N}$ is called \textbf{complex balanced} if every equilibrium concentration $\mathbf{x}^* \in \mathbb{R}_{>0}^n$ is a complex balanced equilibrium concentration.
\end{definition}

Many strong properties are known about complex balanced networks. In particular, it is known that complex balanced networks permit exactly one positive equilibrium concentration in each invariant space of the network and that this equilibrium concentration is locally asymptotically stable (Lemma 4C and Theorem 6A, \cite{H-J1}). Complex balanced systems are also known to be weakly reversible so that they are a subset of the weakly reversible networks considered in Section \ref{wrsection} (Theorem 2B, \cite{H}).

The following result shows complex balancing is a system property depending on the structure and parameters of the network and not on the chosen equilibrium concentration.
\begin{theorem}[Theorem 6A, \cite{H-J1}]
\label{hornjackson}
If a chemical reaction network $\mathcal{N}$ is complex balanced at an equilibrium concentration $\mathbf{x}^* \in \mathbb{R}_{>0}^n$ then it is complex balanced at all of its equilibrium concentrations.
\end{theorem}
\noindent It should be noted that the complex balancing of a network is still dependent on the choice of rate constants. It is possible for a reaction network to be complex balanced for some choices of rate constants and not for others.

In \cite{Sz-H}, an algorithm was presented which was capable of determining sparse and dense complex balanced networks which are trivially linearly conjugate to a given network $\mathcal{N}$. This method required determining an equilibrium value $\mathbf{x}^* \in \mathbb{R}_{>0}^n$ of the network $\mathcal{N}$ and then imposing the condition
\[A_k' \cdot \Psi(\mathbf{x}^*) = \mathbf{0}\]
on $\mathcal{N}'$ in accordance with (\ref{cb}). In this section, we extend these results to include non-trivial linearly conjugate networks.

Suppose that $\mathcal{N}$ and $\mathcal{N}'$ are linearly conjugate via the transformation $\mathbf{y} = T^{-1} \mathbf{x}$. In order to guarantee the network $\mathcal{N}'$ is complex balanced, according to (\ref{cb}) we require that
\begin{equation}
\label{234}
A_k' \cdot \Psi(\mathbf{y}^*) = \mathbf{0}.
\end{equation}
Since the equilibrium concentrations of $\mathcal{N}$ and $\mathcal{N}'$ are related by the transformation $\mathbf{y}^* = T^{-1} \mathbf{x}^*$, we have that the left-hand side of (\ref{234}) can be rewritten
\[A_k' \cdot \Psi(\mathbf{y}^*) = A_k' \cdot \Psi(T^{-1} \mathbf{x}^*) = A_k' \cdot \mbox{diag} \left\{ \Psi(\mathbf{c}) \right\}^{-1} \cdot \Psi(\mathbf{x}^*) = A_b \cdot \Psi(\mathbf{x}^*)\]
where we have made use of the form of the kinetics matrix of $\mathcal{N}'$ according to (\ref{newrateconstants}). The condition for complex balancing of the linearly conjugate network $\mathcal{N}'$ is therefore
\begin{equation}
\label{cb2}
\mbox{\textbf{(CB)}} \; \; \left\{ \; \; \begin{array}{l} A_b \cdot \Psi(\mathbf{x}^*) = \mathbf{0} \\ M \cdot \Psi(\mathbf{x}^*) = \mathbf{0} \\ \mathbf{x}^* \in \mathbb{R}_{>0}^n \end{array} \right.
\end{equation}
where $M$ is as in (\ref{M}). A sparse or dense complex balanced network which is linearly conjugate to $\mathcal{N}$ can be found by optimizing (\ref{sparse}) or (\ref{dense}), respectively, over the constraint sets (\ref{conjugate}), (\ref{density}) and (\ref{cb2}).

It should be noted that the optimization algorithm is less computationally exhausting than the corresponding algorithm for weak reversibility (\ref{weakreversibility}). This is because the matrix $\tilde{A}_k$ required in the general weakly reversible case is not required in the complex balancing condition; rather, it is sufficient to use the matrix $A_b$. Consequently, there are fewer decision variables in the complex balancing algorithm. The pre-step that a $\mathbf{x}^* \in \mathbb{R}_{>0}^n$ be found satisfying $Y \cdot A_k \cdot \Psi(\mathbf{x}^*) = \mathbf{0}$ may off-set this advantage, however, depending on the difficulty in solving $Y \cdot A_k \cdot \Psi(\mathbf{x}^*) = \mathbf{0}$.

It is unclear how the outcome of the algorithm depends on the choice of equilibrium concentration $\mathbf{x}^* \in \mathbb{R}_{>0}^n$, which in general is not unique. The following result clarifies this dependence.

\begin{theorem}
\label{cbtheorem}
Suppose $\mathcal{N}$ is linearly conjugate to $\mathcal{N}'$ with transformation matrix $T = \mbox{diag} \left\{ \mathbf{c} \right\}$ where $\mathbf{c} \in \mathbb{R}_{>0}^n$ and suppose $\mathcal{N}'$ is complex balanced at $\mathbf{y}^* = T^{-1} \mathbf{x}^*$ where $\mathbf{x}^* \in \mathbb{R}_{>0}^n$ and $Y \cdot A_k \cdot \Psi(\mathbf{x}^*) = \mathbf{0}$. Then $\mathcal{N}'$ is complex balanced at $\bar{\mathbf{y}}^* = T^{-1} \bar{\mathbf{x}}^*$ for all $\bar{\mathbf{x}}^* \in \mathbb{R}_{>0}^n$ satisfying $Y \cdot A_k \cdot \Psi(\bar{\mathbf{x}}^*) = \mathbf{0}$.
\end{theorem}

\begin{proof}

Suppose trajectories of $\mathcal{N}'$ are related to trajectories of $\mathcal{N}$ by the relationship $\mathbf{y} = T^{-1} \mathbf{x}$ where $T = \mbox{diag} \left\{ \mathbf{c} \right\}$ and $\mathbf{c} \in \mathbb{R}_{>0}^n$.

Suppose that $\mathcal{N}'$ is complex balanced at $\mathbf{y}^* = T^{-1} \mathbf{x}^*$ where $\mathbf{x}^* \in \mathbb{R}_{>0}^n$ is an equilibrium concentration of $\mathcal{N}$. It follows from Theorem \ref{hornjackson} that
\begin{equation}
\label{8348}
Y \cdot A_k' \cdot \Psi(\mathbf{y}) = \mathbf{0} \; \; \; \Longrightarrow \; \; \; A_k' \cdot \Psi(\mathbf{y}) = \mathbf{0}.
\end{equation}

Now consider an arbitrary equilibrium concentration $\bar{\mathbf{x}}^* \in \mathbb{R}_{>0}^n$ of $\mathcal{N}$. Since $\mathcal{N}$ and $\mathcal{N}'$ are linearly conjugate, it follows that $Y \cdot A_k = T \cdot Y \cdot A_k' \cdot \mbox{diag} \left\{ \Psi(\mathbf{c}) \right\}^{-1}$. It follows that we have
\[\begin{split} \mathbf{0} = Y \cdot A_k \cdot \Psi(\bar{\mathbf{x}}^*) & = T \cdot Y \cdot A_k' \cdot \mbox{diag} \left\{ \Psi(\mathbf{c}) \right\}^{-1} \cdot \Psi(T \bar{\mathbf{y}}^*) \\ & = T \cdot Y \cdot A_k' \cdot \Psi(\bar{\mathbf{y}}^*).\end{split}\]
It follows by the structure of $T$ that $Y \cdot A_k' \cdot \Psi(\bar{\mathbf{y}}^*) = \mathbf{0}$. From (\ref{8348}) we have that $A_k' \cdot \Psi(\bar{\mathbf{y}}^*) = \mathbf{0}$. In other words, $\mathcal{N}'$ is complex balanced at $\bar{\mathbf{y}}^*$ and we are done.

\end{proof}

This result shows that when imposing the complex balancing constraint (\ref{cb2}) on $\mathcal{N}'$, it does not matter which equilibrium concentration of $\mathcal{N}$ we choose. The feasible set of solutions (i.e. admissible networks) is the same.

\subsection{Detailed Balanced Systems}

In \cite{Sz-H}, the authors present an algorithm for determining detailed balanced networks which are trivially linearly conjugate to a given network. In this section we extend this algorithm to non-trivial linear conjugacies.

\begin{definition}
An equilibrium concentration $\mathbf{x}^* \in \mathbb{R}_{>0}^n$ of the chemical reaction network $\mathcal{N}$ is a \textbf{detailed balanced} equilibrium concentration if
\begin{equation}
\label{db}
[A_k]_{ij} \Psi_j(\mathbf{x}^*) = [A_k]_{ji} \Psi_i(\mathbf{x}^*), \; \; \; \forall \; i,j = 1, \ldots, m, \; \; i \not= j.
\end{equation}
The network $\mathcal{N}$ is called \textbf{detailed balanced} if every equilibrium concentration $\mathbf{x}^* \in \mathbb{R}_{>0}^n$ is a complex balanced equilibrium concentration.
\end{definition}
\noindent In other words, an equilibrium concentration $\mathbf{x}^* \in \mathbb{R}_{>0}^n$ is detailed balanced if the flow across each reaction is balanced by the flow across an opposite reaction at $\mathbf{x}^*$.

Suppose that $\mathcal{N}$ and $\mathcal{N}'$ are linearly conjugate via the transformation $\mathbf{y} = T^{-1} \mathbf{x}$. In order to guarantee the network $\mathcal{N}'$ is detailed balanced, according to (\ref{db}) we require that
\begin{equation}
\label{235}
\mbox{diag} \left\{ \Psi(\mathbf{y}^*) \right\} \cdot (A_k')^T = A_k' \cdot \mbox{diag} \left\{ \Psi(\mathbf{y}^*) \right\}.
\end{equation}
Since the equilibrium concentrations of $\mathcal{N}$ and $\mathcal{N}'$ are related by the transformation $\mathbf{y}^* = T^{-1} \mathbf{x}^*$, we have that
\[\begin{split} A_k' \cdot \mbox{diag} \left\{ \Psi(\mathbf{y}^*) \right\} & = A_k' \cdot \mbox{diag} \left\{ \Psi(\mathbf{c}) \right\}^{-1} \cdot \mbox{diag} \left\{ \Psi(\mathbf{x}^*) \right\} \\ & = A_b \cdot \mbox{diag} \left\{ \Psi(\mathbf{x}^*) \right\} \end{split}\]
where we have made use of the form of the kinetics matrix of $\mathcal{N}'$ according to (\ref{newrateconstants}). The condition for detailed balancing of the linearly conjugate network $\mathcal{N}'$ is therefore
\begin{equation}
\label{db2}
\mbox{\textbf{(DB)}} \; \; \left\{ \; \; \begin{array}{l} \mbox{diag} \left\{ \Psi(\mathbf{x}^*) \right\} \cdot A_b^T = A_b \cdot \mbox{diag} \left\{ \Psi(\mathbf{x}^*) \right\} \\ M \cdot \Psi(\mathbf{x}^*) = \mathbf{0} \\ \mathbf{x}^* \in \mathbb{R}_{>0}^n \end{array} \right.
\end{equation}
where $M$ is as in (\ref{M}). A sparse or dense detailed balanced network which is linearly conjugate to $\mathcal{N}$ can be found by optimizing (\ref{sparse}) or (\ref{dense}), respectively, over the constraint sets (\ref{conjugate}), (\ref{density}) and (\ref{db2}). Since the analogous result to Theorem \ref{cbtheorem} holds as a consequence of detailed balanced systems being a subset of complex balanced networks, we do not prove it here.

\subsection{Minimal and Maximal Number of Complexes}

We can also adapt to non-trivial linear conjugacies the algorithm introduced in \cite{Sz-H} for determining a network with the fewest or greatest number of complexes from a fixed complex set which is trivially linearly conjugate to a given network $\mathcal{N}$.

In order to count the number of complexes in the network, we introduce the binary variables $\delta_i \in \left\{ 0, 1 \right\}$, $i, j = 1, \ldots, m$, and consider the logical equations
\begin{equation}
\label{888}
\delta_i = 1 \; \Longleftrightarrow \; \mathop{\sum_{j_1=1}^{m}}_{j_1 \not= i}[A_k]_{ij_1} + \mathop{\sum_{j_2=1}^{m}}_{j_2 \not= i}[A_k]_{j_2i} > 0
\end{equation}
for $i=1, \ldots, m$. In other words, $\delta_i$ takes on the value of one if and only if there is a reaction to or from the complex $\mathcal{C}_i$ in the network; otherwise, it takes the value zero. For computational purposes, we reconsider (\ref{888}) as 
\begin{equation}
\label{889}
\delta_i = 1 \; \Longleftrightarrow \; \mathop{\sum_{j_1=1}^{m}}_{j_1 \not= i}[A_k]_{ij_1} + \mathop{\sum_{j_2=1}^{m}}_{j_2 \not= i}[A_k]_{j_2i} \geq \epsilon
\end{equation}
where $0 < \epsilon \ll 1$. The linear constraints required to substantiate (\ref{889}) are
\begin{equation}
\label{complexes}
\mbox{\textbf{(Comp)}} \; \; \left\{ \; \; \begin{array}{l} \displaystyle{0 \leq \mathop{\sum_{j_1=1}^{m}}_{j_1 \not= i}[A_k]_{ij_1} + \mathop{\sum_{j_2=1}^{m}}_{j_2 \not= i}[A_k]_{j_2i} - \delta_i \epsilon} \\ \displaystyle{0 \leq -\mathop{\sum_{j_1=1}^{m}}_{j_1 \not= i}[A_k]_{ij_1} - \mathop{\sum_{j_2=1}^{m}}_{j_2 \not= i}[A_k]_{j_2i}} \\ \; \; \; \; \; \; \; \; \; \; \; \; \; \; \; \; \; \; \; \; \; \displaystyle{+ \left( \mathop{\sum_{j_1=1}^{m}}_{j_1 \not= i} u_{ij_1} + \mathop{\sum_{j_2=1}^{m}}_{j_2 \not= i} u_{j_2i} \right) \delta_i} \\ \delta_i \in \left\{ 0, 1 \right\}, i = 1, \ldots, m. \end{array} \right.
\end{equation}
We can now determine a network with the fewest or greatest number of complexes by optimizing the functions
\begin{equation}
\label{fewest}
\mbox{\textbf{(Min)}} \; \; \left\{ \; \; \; \; \; \; \; \; \mbox{minimize} \; \; \; \; \; \sum_{i=1}^m \delta_i \right.
\end{equation}
or
\begin{equation}
\label{greatest}
\mbox{\textbf{(Max)}} \; \; \left\{ \; \; \; \; \; \; \; \; \mbox{minimize} \; \; \; \; \;-\sum_{i=1}^m \delta_i \right.
\end{equation}
respectively over the constraint sets (\ref{conjugate}) and (\ref{complexes}). Further constraints can be imposed to restrict ourselves to specific classes of systems (e.g. complex balanced systems, reversible networks, etc.) although care has to be taken to ensure the structural constraints are still satisfied.

\subsection{Examples}

In this section, we present a few examples which illustrate the methodologies outlined so far.\\

\noindent \textbf{Example 1:} Consider the kinetic scheme
\begin{equation}
\label{example1}
\begin{split} \dot{x}_1 & = x_1 x_2^2 - 2 x_1^2 + x_1 x_3^2 \\ \dot{x}_2 & = -x_1^2 x_2^2 + x_1 x_3^2 \\ \dot{x}_3 & = x_1^2 - 3 x_1 x_3^2\end{split}
\end{equation}
introduced in \cite{J-S4}. In that paper, it was shown that the kinetics could be generated by a reaction network involving the complex set
\[\begin{split} & C_1 = X_1 + 2 X_2, C_2 = 2X_1 + 2 X_2, C_3 = 2X_1 + X_2, \\ & C_4 = 2X_1, C_5 = X_1, C_6 = 2X_1 + X_3, C_7 = X_1 + 2X_3 \\ & C_8 = 2X_1 + 2X_3, C_9 = X_1 + X_2 + 2X_3, C_{10} = X_1 + X_3\end{split}\]
and that (\ref{example1}) has dynamics which is linearly conjugate to those generated by the sparse weakly reversible network given in Figure \ref{figure1}(a) (conjugacy constants $c_1 = 20, c_2 = 2, c_3 = 5$) and the dense weakly reversible network given in Figure \ref{figure1}(b) (conjugacy constants $c_1 = 20/3, c_2 = 20/33, c_3=5/3$).

\begin{figure}[h]
\begin{center}
\includegraphics[width=11cm]{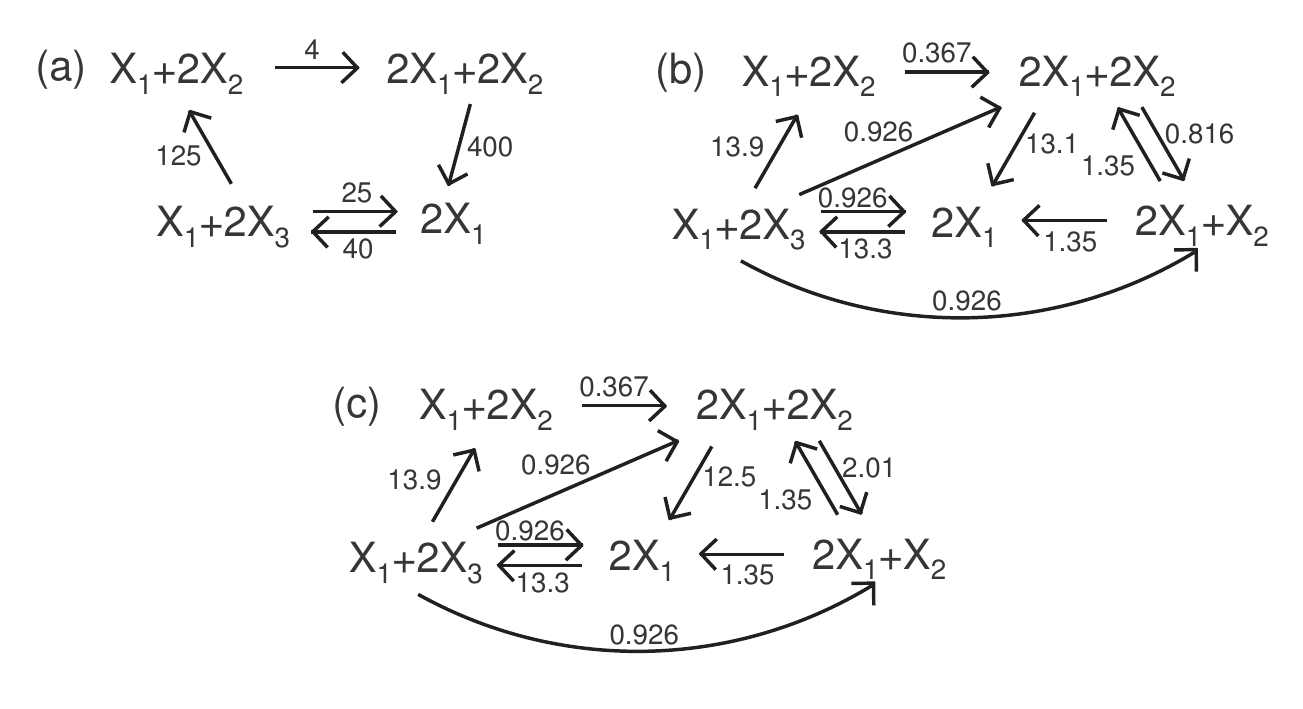}
\end{center}
\vspace{-0.3in}
\caption{Weakly reversible networks which are linearly conjugate to a network with the kinetics (\ref{example1}). The network in (a) is sparse while the networks in (b) and (c) are dense. The networks (a) and (c) are also complex balanced. (The parameter values in (b) and (c) have been rounded to three significant figures.)}
\label{figure1}
\end{figure}

The network in Figure \ref{figure1}(a) is complex balanced as a consequence of it being a zero deficiency weakly reversible network. It can be easily checked, however, that the network in Figure \ref{figure1}(b) is not complex balanced. We might wonder, therefore, what running the algorithm for a dense complex balanced network which is linearly conjugate to a network generating the kinetics (\ref{example1}) would produce.

Numerically, we can determine that an equilibrium concentration of (\ref{example1}) is $(x_1^*,x_2^*,x_3^*) = (0.2,0.577350269,0.258198889)$. Running GLPK for a sparse network (\ref{sparse}) over the constraints (\ref{conjugate}), (\ref{density}) and (\ref{cb2}) gives the network given in Figure \ref{figure1}(c) (conjugacy constants $c_1 = 20/3, c_2 = 20/33, c_3=5/3$). We notice that this network has the same structure as the weakly reversible network in Figure \ref{figure1}(b) and, furthermore, only differs in two rate constant values. It can be checked, however, that (c) is complex balanced while (b) is not.\\

\noindent \textbf{Example 2:} Consider the kinetic scheme
\begin{equation}
\label{example2}
\begin{array}{ll} \dot{x}_1 = -2x_1x_2+2x_3+2x_6 \hspace{0.5in} & \dot{x}_4 = x_3-x_4x_5+x_6 \\ \dot{x}_2 = -x_1x_2+2x_3 & \dot{x}_5 = -2x_4x_5+4x_6 \\ \dot{x}_3 = 2x_1x_2-4x_3 & \dot{x}_6 = x_4x_5 - 2x_6.\end{array} 
\end{equation}
Using the indexing scheme introduced in \cite{H-T} and more recently applied in the papers \cite{Sz-H} and \cite{J-S4} we can construct a chemical reaction network capable of generating the dynamics (\ref{example2}) under the assumption of mass-action kinetics (\ref{de}) which involves the complexes
\[\begin{split}
& C_1 = X_1 + X_2, C_2 = X_2, C_3 = X_3, C_4 = X_1+X_3,\\ & C_5 = X_6, C_6 = X_1 + X_6, C_7 = X_1, C_8 = X_2+X_3,\\ & C_9 = X_1+X_2+X_3, C_{10} = \emptyset, C_{11} = X_3 + X_4 \\ & C_{12} = X_4 + X_5, C_{13} = X_5, C_{14} = X_6, C_{15} = X_4 + X_6,\\ & C_{16} = X_4, C_{17} = X_5 + X_6, C_{18} = X_4 + X_5 + X_6.
\end{split}\]

It seems less than desirable, however, to consider a network of 18 complexes given the simplistic dynamics (\ref{example2}). We might wonder if there is a network with fewer complexes. Optimizing (\ref{fewest}) over the constraint sets (\ref{conjugate}) and (\ref{complexes}) gives the network
\[\mathcal{N}': \; \; \; \; \; \; \; \; \begin{array}{l} \displaystyle{X_2 + X_3 \; \stackrel{1}{\longleftarrow} \; X_1 + X_2 \; \; \mathop{\stackrel{1}{\rightleftarrows}}_{2} \; \; X_3 \; \stackrel{2}{\longrightarrow} \; X_4} \\ \displaystyle{X_5 \; \stackrel{2}{\longleftarrow} \; X_4+X_5 \; \; \mathop{\stackrel{2}{\rightleftarrows}}_{2} \; \; X_6 \; \stackrel{2}{\longrightarrow} \; X_1 + X_6} \end{array}\]
and the conjugacy constants $c_1 = 1, c_2 = 2, c_3 = 2, c_4 = 1, c_5 = 4, c_6 = 2$.

In terms of understanding the qualitative dynamics of \ref{example2}), this network is not particularly insightful. We might notice, however, that the net effect of all the reaction pathways leading out from $X_3$ is to create an $X_2$ and an $X_4$ at the expense of depleting $X_3$. The complex $X_2 + X_4$, however, has not been considered in the procedure. Similarly, the reaction pathways leading out from $X_6$ generate an $X_1$ and an $X_5$ but the complex $X_1 + X_5$ has not been included in the procedure. We might consider appending the procedure, therefore, to include $C_{19} = X_2 + X_4$ and $C_{20} = X_1 + X_5$. Repeating the algorithm in GLPK gives the network
\[\mathcal{N}'': \; \; \; \; \; \; \; \; \begin{array}{l} \displaystyle{X_1 + X_2 \; \; \mathop{\stackrel{1}{\rightleftarrows}}_{2} \; \; X_3 \; \stackrel{2}{\longrightarrow} \; X_2 + X_4} \\ \displaystyle{X_4 + X_5 \; \; \mathop{\stackrel{2}{\rightleftarrows}}_{1} \; \; X_6 \; \stackrel{1}{\longrightarrow} \; X_1 + X_5} \end{array}\]
and the conjugacy constants $c_1 = 1, c_2 = 1, c_3 = 2, c_4 = 1, c_5 = 2, c_6 = 1$. This is easily identified as the enzyme network
\[\begin{array}{l} \displaystyle{S_1 + E \; \; \rightleftarrows \; \; SE \; \longrightarrow \; S_2 + E} \\ \displaystyle{S_2 + F \; \; \rightleftarrows \; \; PF \; \longrightarrow \; S_1 + F}\end{array}\]
where an enzyme $E$ facilitates the transfer of a substrate $S_1$ into another substract $S_2$ and another enzyme $F$ facilitates the transfer back. This network was considered extensively in \cite{A3} and \cite{A-S2}. In particular, it was shown in \cite{A-S2} that for all rate constant values, the network possesses within each invariant space a unique positive equilibrium concentration which is globally asymptotically stable relative to that invariant space. It follows by the properties of linearly conjugate reaction networks that (\ref{example2}) inherits the same qualitative dynamics.

\section{Structural Dynamical Equivalence}

In this section we extend the computation procedure given in Section \ref{extensions} for dynamical equivalence to the case of networks $\mathcal{N}$ which are structurally fixed but have undetermined parameters.

\subsection{Dynamical Equivalence}

The MILP framework outlined so far requires that the rate constants be specified for the network $\mathcal{N}$. Consequently, when we search for networks which are linearly conjugate to a given network $\mathcal{N}$, we are really asking if there are networks which are linearly conjugate \emph{for a given choice of parameter values}.

For networks where the dynamical behaviour is heavily dependent on the chosen rate constants, however, it is possible that certain behaviours are being overlooked by poor rate constant selection. There are networks, for instance, which are known to be linearly conjugate to weakly reversible networks or complex balanced networks for certain values of the rate constants but not for others (see Examples 2 and 3 of \cite{J-S2}). In such cases, if the rate constants are not carefully chosen, the algorithm may overlook these networks and we would not realize that the mechanism shares characteristics with these other networks.


Therefore, we now change the problem setup by fixing only the structure of the initial network $\mathcal{N}$ but not the parameter values. We will show that this problem class can also be casted to the framework of MILP. This would remove the above mentioned limits of using a fully specified initial network model.

The conditions for dynamical equivalence, keeping the entries of both $A_k$ and $A_k'$ general, are
\begin{equation}
\label{realization}
\mbox{\textbf{(DE)}} \; \; \left\{ \; \; \begin{array}{ll} & Y \cdot A_k' = Y \cdot A_k \\ & \displaystyle{\sum_{i=1}^m [A_k']_{ij} = \sum_{i=1}^m [A_k]_{ij} = 0, \; \; \; j=1, \ldots, m} \\ & \displaystyle{[A_k']_{ij} \geq 0, [A_k]_{ij} \geq 0, \; \; \; i,j = 1, \ldots, m, \; i \not= j} \\ & \displaystyle{[A_k']_{ii} \leq 0, [A_k]_{ii} \leq 0, \; \; \; i = 1, \ldots, m} \end{array} \right.
\end{equation}
Note that in \eqref{realization} both the off-diagonal entries of $A_k$ and $A_k'$ are now decision variables.

As in Section \ref{section1}, we want to keep track of the structure of $A_k'$. The conditions corresponding to (\ref{density}) for the matrix $A_k'$ are
\begin{equation}
\label{density3}
\mbox{\textbf{(S2)}} \; \; \left\{ \; \; \begin{array}{ll} & \displaystyle{0 \leq [A_k']_{ij}-\epsilon \delta_{ij}, \; \; \; i,j = 1, \ldots, m, \; \; i \not= j} \\ & \displaystyle{0 \leq -[A_k']_{ij}+u_{ij} \delta_{ij}, \; \; \; i,j = 1, \ldots, m, \; \; i \not= j}  \\ & \displaystyle{\delta_{ij} \in \left\{ 0, 1 \right\}, \; \; \; i,j = 1, \ldots, m, i \not= j,} \end{array} \right.
\end{equation}
As before, the binary variables $\delta_{ij}$ keep track of whether a reaction is in the network $\mathcal{N}'$ or not and thus are capable of counting the number of reactions in $\mathcal{N}'$.

We also, however, want to permit $A_k$ to have a variable rate constant values within a fixed network structure. In order to fix this network structure, we introduce the binary variables $\gamma_{ij} \in \left\{ 0, 1 \right\}$, $i,j = 1, \ldots, m$, $i \not= j$, and the logical equations
\begin{equation}
\label{2342}
\gamma_{ij} = 1 \; \Longleftrightarrow \; [A_k]_{ij} > \epsilon, \; \; \; i,j = 1, \ldots, m, i \not= j
\end{equation}
for some $0 < \epsilon \ll 1$. In other words, the $\gamma_{ij}$'s keep track of whether the reaction $\mathcal{C}_j \to \mathcal{C}_i$ is in the network $\mathcal{N}$. 
The conditions required to allow the entries of $A_k$ to vary independently within this pre-defined structure are
\begin{equation}
\label{independent}
\mbox{\textbf{(Ind)}} \; \; \left\{ \begin{array}{l}
\displaystyle{0 \leq [A_k]_{ij} - \epsilon \gamma_{ij}, \; \; \; i,j = 1, \ldots, m, \; i \not= j} \\
\displaystyle{0 \leq -[A_k]_{ij} + u_{ij} \gamma_{ij}, \; \; \; i,j = 1, \ldots, m, \; i \not=j}
\end{array} \right.
\end{equation}
where
\begin{equation}
\label{gamma}
\gamma_{ij} = \left\{ \begin{array}{ll} 1, \; \; \; \; \; & \mbox{if }(\mathcal{C}_j,\mathcal{C}_i) \in \mathcal{R} \\ 0, & \mbox{otherwise.}\end{array} \right.
\end{equation}
\noindent Further constraints can be implemented to search through subspaces of the parameter spaces for alternative realizations. For instance, if we suspect the reaction rate for the reaction $\mathcal{C}_{j_1} \to \mathcal{C}_{i_1}$ is slaved to that of $\mathcal{C}_{j_2} \to \mathcal{C}_{i_2}$ we can add $[A_k]_{i_1j_1} = [A_k]_{i_2j_2}$ to the procedure, etc.

The conditions (\ref{realization}), (\ref{density3}) and (\ref{independent}) can be combined with the structural conditions for reversibility (\ref{reversible}) and weak reversibility (\ref{weakreversibility} and \ref{density2}) and the objective functions (\ref{sparse}) and (\ref{dense}) to search over the parameter space of $\mathcal{N}$ for sparse and dense alternative realizations $\mathcal{N}'$ which satisfy these further structural constraints.

\subsection{Complex Balanced Realizations}

It is also desirable to explore the parameter space of $\mathcal{N}$ for alternative realizations $\mathcal{N}'$ which are complex balanced. The linear constraints (\ref{db2}) and (\ref{cb2}), however, cannot be used in the parameter-independent case since the required equilibrium concentrations $\mathbf{x}^* \in \mathbb{R}_{>0}^n$ depend on the rate constants for $\mathcal{N}$ which are not specified.

We might expect, however, since all weakly reversible networks are complex balanced for some choice of rate constants, that if a network $\mathcal{N}$ has a weakly reversible alternative realization $\mathcal{N}'$ for some other choice of rate constants then it also has a complex balanced alternative realization $\mathcal{N}''$ for some choice of rate constants. That is to say, in the parameter-independent optimization procedure weak reversibility is sufficient to demonstrate complex balancing. The main result of this subsection guarantees this (Theorem \ref{bigtheorem}).

First, however, we need the following results about weakly reversible networks.

\begin{theorem}[Theorem 3.1, \cite{G-H}; Proposition 4.1, \cite{F3}]
\label{wr}
Let $A_k$ be a kinetics matrix and let $\Lambda_i$, $i=1, \ldots, \ell,$ denote the support of the $i^{th}$ linkage class. Then the reaction graph corresponding to $A_k$ is weakly reversible if and only if there is a basis of ker$(A_k)$, $\left\{ \mathbf{b}^{(1)}, \ldots, \mathbf{b}^{(\ell)} \right\}$, such that, for $i=1, \ldots, \ell$,
\[\mathbf{b}^{(i)} = \left\{ \begin{array}{ll} b^{(i)}_j > 0, \hspace{0.3in} & j \in \Lambda_i \\ b^{(i)}_j = 0, & j \not\in \Lambda_i. \end{array} \right.\]
\end{theorem}

\begin{theorem}[Theorem 1, \cite{D-F-J-N}]
\label{feinberg}
Under the assumption of mass-action kinetics, weakly reversible chemical reaction networks possess at least one positive equilibrium concentration within each positive invariant space of the system.
\end{theorem}

An immediate consequence of Theorem \ref{wr} is that a network is weakly reversible if and only if there is a vector $\mathbf{b} \in \mathbb{R}_{>0}^n$ in the kernel of $A_k$. We will exploit this fact in the next result.

\begin{theorem}
\label{bigtheorem}
Suppose there is a choice of rate constants such that the network $\mathcal{N}$ is dynamically equivalent to $\mathcal{N}'$ and $\mathcal{N}'$ is weakly reversible. Then there exists a choice of rate constants such that the network $\mathcal{N}$ is dynamically equivalent to $\mathcal{N}''$ where $\mathcal{N}''$ is complex balanced. Furthermore $\mathcal{N}''$ can be selected to have the same structure as $\mathcal{N}'$.
\end{theorem}

\begin{proof}
Let $A_k$ be the kinetics matrix associated with $\mathcal{N}$ and $A_k'$ be the kinetics matrix associated with $\mathcal{N}'$, and suppose that $\mathcal{N}'$ is weakly reversible. Let $\mathbf{b} \in \mathbb{R}_{>0}^n$ denote the positive vector in ker$(A_k)$ guaranteed to exist by Theorem \ref{wr} and $\mathbf{x}^* \in \mathbb{R}_{>0}^n$ be any positive equilibrium concentration of (\ref{de}) guaranteed to exist by Theorem \ref{feinberg}.

We now define a new network $\mathcal{N}''$ with the associated kinetics matrix
\begin{equation}
\label{construct}
A_k'' = A_k' \cdot \mbox{diag} \left\{ \frac{\mathbf{b}}{\Psi(\mathbf{x}^*)} \right\}
\end{equation}
where we define vector division to be componentwise division, i.e. $\mathbf{x} / \mathbf{y} = \left[ x_1/y_1, x_2/y_2, \ldots, x_n/y_n \right]$ for $\mathbf{x}, \mathbf{y} \in \mathbb{R}^n$. Notice that all the terms in the definition of $A_k''$ can be determined under the assumption that $\mathcal{N}'$ is weakly reversible. We have that
\[A_k'' \cdot \Psi(\mathbf{x}^*) = A_k' \cdot \mathbf{b} = \mathbf{0}\]
since $\mathbf{b} \in $ ker$(A_k')$ so that $\mathcal{N}''$ is complex balanced at $\mathbf{x}^* \in \mathbb{R}_{>0}^m$. Furthermore, we have that
\[Y \cdot A_k'' = Y \cdot A_k' \cdot \mbox{diag} \left\{ \frac{\mathbf{b}}{\Psi(\mathbf{x}^*)} \right\} = Y \cdot A_k \cdot \mbox{diag} \left\{ \frac{\mathbf{b}}{\Psi(\mathbf{x}^*)} \right\}.\]
It is clear that $A_k \cdot$ diag$\left\{ \mathbf{b} / \Psi(\mathbf{x}^*) \right\}$ has the same structure as $A_k$ so that this corresponds to different choice of rate constants for the network $\mathcal{N}$. The result follows.
\end{proof}

Note that, while we do not have linear constraints capable of determining complex balanced networks explicitly, we can always construct a complex balanced network from a weakly reversible or reversible network output by the optimization procedure. (Although this may not hold if further restrictions on the parameter space of $\mathcal{N}$ have been imposed.)

It is worth noting that the corresponding result for reversible and detailed balanced networks does not follow in the same manner as the proof of Theorem \ref{bigtheorem}. This is because, for reversible networks with cycles, it is known that the detailed balancing condition entails further conditions on the rate constants above and beyond complex balancing \cite{Fe,D-M}. We can still, however, construct a complex balanced realization from an arbitrary reversible network according to Theorem \ref{bigtheorem}. Since detailed balancing implies no further dynamical information above and beyond complex balancing, for all practical purposes this is as far as we need to go.

It is also worth noting that this algorithm cannot determine networks which are linear conjugate to a given structurally-fixed network; it can only find dynamically equivalent networks. This is clear since linear conjugacy requires that
\begin{equation}
\label{123}
Y \cdot A_k = T \cdot Y \cdot A_b
\end{equation}
where $T = \mbox{diag} \left\{ \mathbf{c} \right\}$ and $\mathbf{c} \in \mathbb{R}_{>0}^m$. Regardless of whether we consider the transformation $T$ on the left-hand-side or right-hand-side of (\ref{123}), it produces a non-linear condition and therefore cannot currently be placed within the existing MILP framework.\\

\subsection{Examples}

In this section, we introduce a few examples which illustrate how the algorithm for producing weakly reversible, reversible, and complex balanced realizations of a structurally-fixed but parameter-variable network works.\\

\noindent \textbf{Example 3:} Consider the reaction network $\mathcal{N}$ given by
\[\mathcal{N}: \; \; \; \; \; \; \; \; 2X_1 + X_2 \; \stackrel{\alpha}{\longrightarrow} \; \; 3X_1 \; \; \mathop{\stackrel{1}{\rightleftarrows}}_{1} \; \; 3X_2 \; \; \stackrel{\alpha}{\longleftarrow} \; X_1 + 2X_2\]
where $\alpha > 0$. Despite the reversible step in the central reaction, $\mathcal{N}$ is neither fully nor weakly reversible and therefore the dynamics do not fall within the scope of the theory of such networks.

We want to check whether there are weakly or fully reversible networks which are dynamically equivalent to $\mathcal{N}$ for some value of $\alpha$. We set $C_1 = 2X_1+X_2$, $C_2 = 3X_1$, $C_3 = 3X_2$, and $C_4 = X_1 + 2X_2$. Searching for a sparse network (\ref{sparse}) in GLPK over the weakly reversible constraints (\ref{realization}), (\ref{density3}), (\ref{independent}), (\ref{weakreversibility}) and (\ref{density2}), with the additional constraints $[A_k]_{21} = [A_k]_{34}$ and $[A_k]_{23}=[A_k]_{32}=1$ and the bounds $\epsilon = 1/20$ and $u_{ij} = 20$ $i,j=1, \ldots 4$, $i \not= j$, gives the alternative realization
\[ \mathcal{N}': \; \; \; \; \; \; \; \; \begin{array}{c} \; 2X_1 + X_2 \; \stackrel{1/20}{\longrightarrow} \; 3X_1 \\ {}^{3/2} \uparrow \; \; \; \; \; \; \; \; \; \; \; \; \; \; \; \; \; \; \; \; \; \;  \; \; \downarrow_{3/2} \\ \; \; 3X_2 \; \stackrel{1/20}{\longleftarrow} \; X_1 + 2X_2 \end{array}\]
and $\alpha = 1/20$ for the original network $\mathcal{N}$. The network $\mathcal{N}'$ has the corresponding kinetics matrix
\[A_k' = \left[ \begin{array}{cccc} -\frac{1}{20} & 0 & \frac{3}{2} & 0 \\ \frac{1}{20} & -\frac{3}{2} & 0 & 0 \\ 0 & 0 & -\frac{3}{2} & \frac{1}{20} \\ 0 & \frac{3}{2} & 0 & -\frac{1}{20} \end{array} \right]\]
and the positive equilibrium concentration $(x_1^*,x_2^*)=(1,1)$. It can easily be checked that $\mathcal{N}'$ is not complex balanced at this equilibrium concentration.

In order to construct a complex balanced network $\mathcal{N}''$ with the same structure as $\mathcal{N}'$ by (\ref{construct}) we need to determine a vector $\mathbf{b} \in \mathbb{R}_{>0}^4$ such that $\mathbf{b} \in $ ker$(A_k')$. It can be easily checked that the vector $\mathbf{b} = [ 3/2 \; \; 1/20 \; \; 1/20 \; \; 3/2]^T$ works; however, using this choice of $\mathbf{b}$ produces a set of rate constants by 
\[A_k \cdot \mbox{diag} \left\{ \frac{\mathbf{b}}{\Psi(\mathbf{x}^*)} \right\}\]
which clearly violates the condition $[A_k]_{23}=[A_k]_{32}=1$. This can be solved by choosing another multiple of $\mathbf{b}$. In fact, we can see that for the vector $\mathbf{b} = [ 30 \; 1 \; 1 \; 30]^T$ the appropriate rate constant choices for $\mathcal{N}$ occur by choosing $\alpha = 3/2$ and the corresponding complex balanced realization given by (\ref{construct}) is
\[ \mathcal{N}'': \; \; \; \; \; \; \; \; \begin{array}{c} \; 2X_1 + X_2 \; \stackrel{3/2}{\longrightarrow} \; 3X_1 \\ {}^{3/2} \uparrow \; \; \; \; \; \; \; \; \; \; \; \; \; \; \; \; \; \; \; \; \; \;  \; \; \downarrow_{\; 3/2} \\ \; \; 3X_2 \; \stackrel{3/2}{\longleftarrow} \; X_1 + 2X_2 \end{array}\]
(This corresponds to a scaling by $3/2$ of the `block' network given in \cite{H-J1}. From the analysis presented in that paper, we know the network is complex balanced only for this particular value of $\alpha$.)

We may also be interested in fully reversible alternative realizations of $\mathcal{N}$. Replacing the constraints (\ref{weakreversibility}) and (\ref{density2}) in the above algorithm by (\ref{reversible}) and running in GLPK gives the network
\[\mathcal{N}': \; \; \; \; \; \; \; \; 2X_1 + X_2 \; \mathop{\stackrel{1/20}{\rightleftarrows}}_{3} \; 3X_1, \; \; \; \; \; \; \; \; X_1 + 2 X_2 \; \mathop{\stackrel{1/20}{\rightleftarrows}}_{3} \; 3X_2\]
and $\alpha=1/20$ for the network $\mathcal{N}$. The corresponding kinetics matrix is
\[A_k' = \left[ \begin{array}{cccc} -\frac{1}{20} & 3 & 0 & 0 \\ \frac{1}{20} & -3 & 0 & 0 \\ 0 & 0 & -3 & \frac{1}{20} \\ 0 & 0 & 3 & -\frac{1}{20} \end{array} \right]\]
which has the positive equilibrium concentration $(x_1^*,x_2^*)=(1,1)$ which is neither complex nor detailed balanced. In order to find a complex balanced network with the same structure as $\mathcal{N}'$ we notice that the kernel of $A_k'$ is given by the span of $[ 3 \; \; 1/20 \; \; 0 \; \; 0]^T$ and $[ 0 \; \; 0 \; \; 1/20 \; \; 3]^T$. In order to preserve the property $[A_k]_{23}=[A_k]_{32}=1$ for the network $\mathcal{N}$ we need to choose $\mathbf{b} = [60 \; 1 \; 1 \; 60]^T$. This gives the value of $\alpha = 3$ for the network $\mathcal{N}$ and the complex balanced realization
\[\mathcal{N}'': \; \; \; \; \; \; \; \; 2X_1 + X_2 \; \mathop{\stackrel{3}{\rightleftarrows}}_{3} \; 3X_1, \; \; \; \; \; \; \; \; X_1 + 2 X_2 \; \mathop{\stackrel{3}{\rightleftarrows}}_{3} \; 3X_2.\]
It can easily be checked that $\mathcal{N}''$ is also detailed balanced.\\

\noindent \textbf{Example 4:} Consider the reaction network $\mathcal{N}$ given by
\[\mathcal{N}: \; \; \; \; \; \; \; \; 2X_1 \; \stackrel{1}{\longrightarrow} \; X_1 + X_2 \; \stackrel{1}{\longleftarrow} \; 2X_2\]
and the reversible alternative realization $\mathcal{N}'$ given by
\[\mathcal{N}': \; \; \; \; \; \; \; \; 2 X_1 \; \; \mathop{\stackrel{1/2}{\rightleftarrows}}_{1/2} \; \; X_1 + X_2 \; \; \mathop{\stackrel{1/2}{\rightleftarrows}}_{3/4} \; \; 2X_2 \; \; \mathop{\stackrel{1/8}{\rightleftarrows}}_{1/4} \; \; 2X_1.\]

If we make the associations $C_1 = 2X_1$, $C_2 = X_1 + X_2$ and $C_3 = 2X_2$, then the network $\mathcal{N}'$ has the corresponding kinetics matrix
\[A_k' = \left[ \begin{array}{ccc} -\frac{3}{4} & \frac{1}{2} & \frac{1}{8} \\ \frac{1}{2} & -1 & \frac{3}{4} \\ \frac{1}{4} & \frac{1}{2} & -\frac{7}{8} \end{array} \right].\]
Choosing any equilibrium concentration $\mathbf{x}^* \in \mathbb{R}_{>0}^2$ we have $A_k' \cdot \Psi(\mathbf{x}^*) \not= \mathbf{0}$ so that the network is not complex balanced.

We wish to construct a complex balanced realization $\mathcal{N}''$ using the methodology outlined in the proof of Theorem \ref{bigtheorem}. We choose the equilibrium value $(x_1^*,x_2^*)=(1,1)$ so that $\Psi(\mathbf{x}^*) = [1 \; \; 1 \; \; 1]^T$ and the kernel vector $\mathbf{b} = [1 \; \; 5/4 \; \; 1]^T$. According to (\ref{construct}) we have
\[A_k'' =  \left[ \begin{array}{ccc} -\frac{3}{4} & \frac{1}{2} & \frac{1}{8} \\ \frac{1}{2} & -1 & \frac{3}{4} \\ \frac{1}{4} & \frac{1}{2} & -\frac{7}{8} \end{array} \right] \left[ \begin{array}{ccc} 1 & 0 & 0 \\ 0 & \frac{5}{4} & 0 \\ 0 & 0 & 1 \end{array} \right] = \left[ \begin{array}{ccc} -\frac{3}{4} & \frac{5}{8} & \frac{1}{8} \\ \frac{1}{2} & -\frac{5}{4} & \frac{3}{4} \\ \frac{1}{4} & \frac{5}{8} & -\frac{7}{8} \end{array} \right].\]
It can be easily checked that the corresponding network $\mathcal{N}''$ is complex balanced and is dynamically equivalent to $\mathcal{N}$. (In general the rate constants of $\mathcal{N}$ may change; however, in this case we have $A_k \cdot $ diag$\left\{ \mathbf{b} / \Psi(\mathbf{x}^*) \right\} = A_k$.)

Despite being complex balanced and reversible, the realization $\mathcal{N}''$ is not detailed balanced according to (\ref{db}). We might wonder if we can construct an alternative realization in a similar manner as (\ref{construct}). We can easily see, however, that applying the detailed balancing conditions to $A_k \cdot $ diag$\left\{ \mathbf{c} \right\}$, where $\mathbf{c} \in \mathbb{R}_{>0}^3$ is arbitrary, produces the unsatisfiable set of conditions
\[\frac{c_1}{2} x_1^2 = \frac{c_2}{2} x_1 x_2, \; \; \; \; \; \frac{c_1}{4} x_1^2 = \frac{c_3}{8} x_2^2, \; \; \; \; \; \frac{c_2}{2} x_1 x_2 = \frac{3c_3}{4} x_2^2.\]
In other words, we cannot always construct a detailed balanced network from an arbitrary reversible network $\mathcal{N}'$ in the same constructive manner used for complex balanced networks from weakly reversible networks.

\section{Conclusions and Future Work}

In this paper, new computational methods were presented for finding networks which are linearly conjugate and dynamically equivalent to a given chemical reaction network endowed with mass-action kinetics.

It was demonstrated that finding a dense or sparse reversible, detailed balanced and complex balanced network which is linearly conjugate to a given network can be framed as a MILP optimization problem. The case of determining conjugate networks which have the greatest and fewest number of complexes was also extended to the case of non-trivial linear conjugacies.

It was shown that, similarly to the case of dynamical equivalence \cite{Sz-H-P}, the graph structure of linearly conjugate dense networks containing the maximal number of nonzero reaction rate coefficients is unique, and that the unweighted directed reaction graph of any linearly conjugate network is a proper subgraph of the unweighted directed reaction graph of the dense one if the set of complexes is given. Moreover, it was proved that arbitrary equilibrium points of the initial network can be used for the existence checking and computation of linearly conjugate complex balanced and detailed balanced networks.

Additionally, the problem of dynamical equivalence was studied when only the structure of the initial network $\mathcal{N}$ is fixed, but its rate constants can take any positive value. It was shown that in this case, the computation of dense and sparse weakly reversible and reversible networks $\mathcal{N}'$ can also be formulated as a MILP problem. It was also shown that complex balanced networks $\mathcal{N}''$ with identical structure to $\mathcal{N}'$ can be constructed from these realizations and corresponded to an alternative choice of rate constants for $\mathcal{N}$. This modification allows us to scan through a range of parameters as part of the procedure and therefore answer questions about a network based on its structure alone. The operation of the developed methods were illustrated on numerical examples. The achievements further extend the applicability range of many existing structure-dependent results of chemical reaction network theory.

Further areas of research and open questions include:
\begin{enumerate}
\item
The procedure introduced to determine structurally dynamically equivalence networks is as of yet unable to search through non-trivial linearly conjugate networks in a linear manner. As such, many networks with potentially insightful information about the dynamics of a given network are being overlooked. Incorporating non-trivial linear conjugacy into a manageable optimization framework is therefore of primary interest.
\item
The results obtained so far depend on the reaction networks under consideration having mass-action kinetics (\ref{de}). Conjugacy and computational results for other widely-used kinetic schemes (e.g. Michaelis-Menten kinetics \cite{M-M}, Hill kinetics \cite{Hi}) would greatly expand the scope of applicability of these methods.
\item
There are many classes of networks with known behaviour lying outside the scope of weakly reversible network theory \cite{A-S2,C-F1,C-F2}. Determining constraints which could restrict our search to within these classes of networks would broaden the scope of dynamical behaviours (e.g. periodicity, oscillatory behaviour, multistability, etc.) we could guarantee through this computational procedure.
\end{enumerate}

\noindent \textbf{Acknowledgements:} M. Johnston and D. Siegel acknowledge the support of D. Siegel's Natural Sciences and Engineering Research Council of Canada Discovery Grant.  G. Szederk\'{e}nyi acknowledges the support of the Hungarian National Research Fund through grant no. OTKA K-83440 as well as the support of project CAFE (Computer Aided Process for Food Engineering)  FP7-KBBE-2007-1
(Grant no: 212754).

\bibliographystyle{plain}
\bibliography{myrefs}

\end{document}